\documentclass[12pt]{amsart} 
\usepackage{amscd}
\usepackage{amssymb}
\usepackage{a4wide}
\usepackage{amstext}
\usepackage{amsthm}
\usepackage{xcolor}
\usepackage{cite}
\usepackage[T1,T2A]{fontenc}
\usepackage[utf8]{inputenc}
\usepackage{url}
\usepackage{amsfonts}
\usepackage{amssymb, amsthm}
\usepackage{amsmath}
\usepackage{mathtools}
\usepackage{needspace}
\usepackage[pdftex]{graphicx}
\usepackage{hyperref}
\usepackage{datetime}
\usepackage{epigraph}
\usepackage{verbatim}
\usepackage{mathtools}
\usepackage{xcolor}
\numberwithin{equation}{section}

\newcommand{\Z}{\mathbb{Z}}

\newcommand{\N}{\mathbb{N}}
\newcommand{\R}{\mathbb{R}}
\newcommand{\Q}{\mathbb{Q}}

\newcommand{\intt}{\int\limits}
\newcommand{\summ}{\sum\limits }

\newcommand{\eps}{\varepsilon}

\newcommand{\CC}{\mathbb {C}}

\newcommand{\RR}{\mathbb{R}}

\DeclareMathOperator{\beqq}{\begin{equation}} 

\DeclareMathOperator{\eeqq}{\end{equation}}

\renewcommand{\phi}{\varphi}

\newcommand{\ZZ}{\mathbb{Z}}

\newcommand{\rl}{\mathbb{R}}

\newcommand{\Lt}{L^2(\RR)}

\newcommand{\beq}{\begin{equation}}
\newcommand{\eeq}{\end{equation}}
\newcommand{\Fg}{\mathcal{F}}

\newcommand{\gG}{\mathfrak{G}} 

\newtheorem{Thm}{Theorem}[section]
\newtheorem{theorem}[Thm]{Theorem}

\newtheorem{lemma}[Thm]{Lemma}
\newtheorem{proposition}[Thm]{Proposition}
\newtheorem{corollary}[Thm]{Corollary}
\newtheorem{remark}[Thm]{Remark}

\newtheorem{definition}{Definition}

\begin{document}

\sloppy
\title[Gabor frames for rational functions]
{Gabor frames for rational functions}

\author{Yurii Belov}
\address{Yurii Belov,
\newline St.~Petersburg State University, St. Petersburg, Russia,
\newline {\tt j\_b\_juri\_belov@mail.ru} }
\author{Aleksei Kulikov}
\address{Aleksei Kulikov,
\newline St.~Petersburg State University, St. Petersburg, Russia,
\newline Department of Mathematical Sciences, Norwegian University of Science and Technology, NO-7491 Trondheim, Norway,
\newline {\tt lyosha.kulikov@mail.ru} }
\author{Yurii Lyubarskii}
\address{Yurii Lyubarskii,
\newline St.~Petersburg State University, St. Petersburg, Russia,
\newline Department of Mathematical Sciences, Norwegian University of Science and Technology, NO-7491 Trondheim, Norway,
\newline {\tt yuralyu@gmail.com} }

\thanks{The work was supported by the Russian Science Foundation grant 19-11-00058 and by Grant 275113 of the Research Council of Norway.}

\begin{abstract} 
We study the frame properties of the Gabor systems
$$\gG(g;\alpha,\beta):=\{e^{2\pi i \beta m x}g(x-\alpha n)\}_{m,n\in\mathbb{Z}}.$$ 
In particular, we prove that for Herglotz windows $g$ such systems always form a frame for $L^2(\mathbb{R})$
if $\alpha,\beta>0$, $\alpha\beta\leq1$. For general rational windows $g\in L^2(\mathbb{R})$ we prove that $\gG(g;\alpha,\beta)$
is a frame for $L^2(\mathbb{R})$ if $0<\alpha,\beta$, $\alpha\beta<1$, $\alpha\beta\not\in\mathbb{Q}$ and
$\hat{g}(\xi)\neq0$, $\xi>0$, thus confirming Daubechies conjecture for this class of functions. We also discuss
some related questions, in particular sampling in shift-invariant subspaces of $L^2(\mathbb{R})$.

\end{abstract}

\maketitle

\section{Introduction and main results}

   We  investigate the Gabor systems generated by linear combinations of the 
   Cauchy kernels, i.e. by the windows of the form   
\beq
\label{eq:0_1}
g(t)=\sum_{k=1}^N \frac {a_k}{t-iw_k}.
\eeq

We describe a new wide class of such functions for which the corresponding Gabor systems possess the frame property for all rectangular
lattices of density at least one. We also observe that for general rational windows of the form \eqref{eq:0_1}, the frame
property of the Gabor system $\{e^{2\pi i \beta m x}g(x-\alpha n)\}_{m,n\in\mathbb{Z}}$ depends on the rationality
of the product $\alpha\beta$, give a precise estimate how large the density $(\alpha\beta)^{-1}$ should be to garantee the frame property
of  the Gabor system and consider sampling in shift-invariant spaces, generated by the window $g$, as well as some
related matters.

One of the central motives of the article is hinted by the Daubechies conjecture \cite[p. 981]{Daub} which assumes that Gabor system is a frame for all $\alpha\beta<1$ whenever $g$ is positive function with positive Fourier transform. This conjecture has been disproved by Janssen \cite{Jans4}, yet in all known examples of functions which generates a Gabor system for all $\alpha\beta<1$ we encounter some kind of positivity.

%

\subsection{Gabor systems}  Given a function $g\in \Lt$  by $ \pi_{x,y}g$ we denote its time frequency shifts 
\beq
\label{eq:01}
\pi_{x,y}g(t)= e^{2\pi i yt}g(t-x), \, x,y\in \RR.
\eeq
For $\alpha,  \, \beta >0$ consider the Gabor system
\beq
\label{eq:02}
\gG(g; \alpha, \beta) = \{\pi_{\alpha m, \beta n}g; \, m,n\in \ZZ\}.
\eeq 

We say that $\gG(g; \alpha, \beta) $ is a  {\em frame} in $\Lt$ if the frame inequality


\beq
\label{eq:03}
A\|f\|^2 \leq \sum_{m,n}  \left | \langle f, \pi_{\alpha m, \beta n} g \rangle \right |^2 \leq B\|f\|^2, \ f\in \Lt
\eeq 
holds  for some $A>0$ and $B<\infty$. 

Gabor systems have been widely used   in signal analysis  and quantum mechanics because of time-frequency localization  
of its elements  $ \pi_{\alpha m, \beta n} g$. For sufficiently dense lattices $\alpha\mathbb{Z}\times \beta\mathbb{Z}$ the supports 
of $ \pi_{\alpha m, \beta n} g$ "cover" $ $ the whole time-frequency plane and
the frame inequality \eqref{eq:03} provides stable reconstruction of a signal $f$  from the inner products
$\langle f, \pi_{\alpha m, \beta n} g \rangle $. On the other hand $\gG(g; \alpha, \beta)$ never forms a frame if $\alpha\beta>1$ (see e.g. \cite{Gro}).
We refer the reader to \cite{Daub, Heil, Gro}, for the detailed history, setting  and discussion
of the problem.

\subsection{Frame set} The fundamental problem of the Gabor analysis is to describe the {\it frame set} of the window $g$:

$$
\Fg(g)= \{(\alpha,  \ \beta); \alpha, \, \beta  > 0 \ {\rm {and}} \ \gG(g; \alpha, \beta)  \ {\rm {is \ a \ frame \ in} } \ \Lt\}.
$$

If $\alpha\beta=1$ complete characterization of the frame set can be given in terms of Zak tarnsform $\mathcal{Z}g$ of the window $g$ 
(see e.g. \cite[Ch.8]{Gro}) but for $\alpha\beta<1$ the frame set $\Fg(g)$ may be very complicated  even for elementary functions $g$ see e.g. \cite{DS, He}. Even the simpler question: for which $g$ does  
 $\Fg(g)$ contain the whole set  $\Pi:= \{ (\alpha, \beta); \alpha, \beta >0, \  \alpha \beta <1 \} $ is also very difficult.

\smallskip
 
The answer has been obtained  for the Gaussian $e^{-x^2}$ \cite{L, S, SW}, truncated $\chi_{(0,\infty)}(x)e^{-x}$ and symmetric $e^{-|x|}$ exponential functions \cite{Jans2,Jans3} , the hyperbolic secant  $(e^x+e^{-x})^{-1}$ \cite{JansStr}. Despite numerous efforts  very little progress has been done  until 2011. A breakthrough  was achieved in  \cite{Gro2}   and  later in \cite{Gro1} where the authors 
 considered the class of  totally positive  functions of finite type and, by using another approach, Gaussian totally positive functions of finite type.   
These results can be viewed as  a contribution to the original conjecture of Daubechies which relates the frame property to the positivity of function and its Fourier transform.  Our results are to large extend motivated by \cite{Gro2} since the Fourier transforms of totally positive functions of finite type have the form $g(t)=P(t)^{-1}$, where  $P$ is a polynomial with simple zeroes located on the imaginary axis, such functions of course admit representation \eqref{eq:0_1}.  

\subsection{Herglotz functions} 

We suggest another approach based on techniques on interpolation by entire functions and dynamical systems. This approach allows us to describe the
frame set for Herglotz functions, study the frames with irrational densities as well as some other special cases.


By Herglotz function we mean a function of the form \eqref{eq:0_1} for which $a_k>0$. Such functions appear naturally in the spectral theory of the Jacobi matrices and the Shroedinger equations.
This class is in a sense opposite to  the class of totally positive functions: while the coefficients $a_k$ in the representation \eqref{eq:0_1}  of the totally positive functions have interlacing signs (and also satisfy a number of additional relations), they are just positive in  the case   of  Herglotz functions.  We will consider Herglotz functions with poles in the upper half-plane. It seems that we encountered another kind of positivity related to the Gabor frame property. 
 
\subsection{Main results} 
  
\begin{theorem}
Let $g$ 
be a  Herglotz function 
\beq
\label{Herg}
g(t)=\sum_{k=1}^N \frac {a_k}{t-iw_k},\quad a_k>0, w_k>0.
\eeq 
  Then 
\begin{equation}
\Fg(g)= \{(\alpha,  \ \beta);  \alpha\beta\leq 1\}.
\label{FS}
\end{equation}
\label{posth}
\end{theorem}

For the general function of the form \eqref{eq:0_1} relation \eqref{FS} does not hold generally speaking. Amazingly
we { \it almost always} have the frame property if $\alpha\beta\not\in\mathbb{Q}$:

\begin{theorem}
Let $g$ be of the form \eqref{eq:0_1} and be such that $m_0(\xi)=\sum_{k=1}^N a_k e^{2\pi \xi w_k}\neq0$, $\xi>0$ and $\Re w_k\neq\Re w_l$ for $k\neq l$. Then  the Gabor system  $\gG(g; \alpha, \beta)$   is a frame in $\Lt$ for any $(\alpha, \beta)\in \Pi$ such that $\alpha\beta\not\in \Q$. 
\label{irrational}
\end{theorem}

Observe that, if $\Re w_k<0$ for all $k$, we have   $m_0(\xi)=\hat g(\xi)$,  $\xi>0$.  
We normalize the Fourier transform as
$$
\hat g(\xi) =\int_{-\infty}^\infty g(t) e^{-2\pi i t \xi} dt.
$$

The next result is an important particular case of Theorem \ref{irrational}.

\begin{theorem}
Let $g$ 
be a  function of the form  \eqref{eq:0_1},  $\Re w_k <0$, $k=1,2, \ \ldots \ , N$, $\Re w_k \neq \Re w_j$ for $j\neq k$  and also 
$ \hat g (\xi) \neq 0$ for 
 $\xi>0$. 
 Then $\gG(g; \alpha, \beta)$ is a frame in $\Lt$ for all $(\alpha, \beta) \in \Pi$, such that $\alpha\beta \not\in \Q$.
\label{irrth}
\end{theorem}

So, for the given class of functions Daubechies conjecture holds literally. Later we will see that the assumption
$\alpha\beta\not \in \mathbb{Q}$ cannot be dropped generally speaking.

\subsection{Near the critical hyperbola} Another interesting question is related to the frame property of $\gG(g; \alpha, \beta)$ when the point $(\alpha,\beta)$ approaches the critical hyperbola $\alpha\beta=1$.

\smallskip

Let $g$ be of the form \eqref{eq:0_1}. Consider the function
$$\mathcal{Z}(z,\xi)=\sum_{k=1}^{N}\frac{a_ke^{2\pi\xi w_k}}{1-ze^{2\pi w_k\slash\alpha}}.$$
In case $\Re w_k<0$, $k=1,...,N$ this function coincides with the classical Zak transform of $g$ up to a non-zero factor.

\begin{theorem}
Let $\Re w_k>0$, $k=1,...,N$  and also $\Re \mathcal{Z}(e^{2\pi i t},\xi)>0$ for all $t\in\mathbb{R}$, $\xi\in\mathbb{R}$. Then there exists $\alpha_0<1$ such that $\gG(g;\alpha,1)$ is a frame in $L^2(\mathbb{R})$ for all $\alpha\in(\alpha_0,1]$.
\label{close1th}
\end{theorem}

By the natural renormalization we have
$$\gG(g(\cdot\slash\beta);\alpha,\beta) \text{ is a frame in } L^2(\mathbb{R})$$
if $\alpha\beta\in(\alpha_0,1]$.

\smallskip
Combining this statement with Theorem \ref{close1th} and some of its corollaries one can obtain.

\begin{theorem}
Let $g$ be of the form \eqref{eq:0_1}, $\Re w_k<0$, $k=1,...,N$,  $\hat{g}(\xi)\neq0$, $\xi>0$, $\Re w_k \neq \Re w_j$ for $j\neq k$ and also
$$\Re\mathcal{Z}(e^{2\pi i t},\xi)>0,\quad  t\in\mathbb{R}, \xi\in\mathbb{R}$$
holds. Then $\gG(g;\alpha,1)$ is a frame in $L^2(\mathbb{R})$ for all $\alpha\in(0,1]$, except  perhaps a finite number of 
exceptional values.
\label{close2th}
\end{theorem}

We also want to highlight the following corollary of Theorem \ref{close2th}.
\begin{corollary} Let $g$ be of the form \eqref{eq:0_1}, $w_k<0$, $a_k\in\mathbb{R}$. If $\hat{g}$ is positive, decreasing, convex function on the positive semiaxis $\mathbb{R}_+$, then $\gG(g;\alpha,\beta)$ is a frame for any pair $(\alpha,\beta)$ sufficiently close
 to the critical hyperbola $\alpha\beta=1$, (i.e. for $\alpha>\alpha(\beta)$).
\label{corrconvex}
\end{corollary}

\subsection{Large densities} Previous theorems deal with the Gabor frames generated by rational functions. 
On the other hand it is known that there exists non-frame rational Gabor systems, in particular, if $g(t)=-\overline{g(-t)}$ and $g(t)=O(t^{-2})$, $t\rightarrow\infty$, the  system $\gG(g;\alpha,\beta)$ does not constitute a frame
in $L^2(\mathbb{R})$ for $\alpha\beta=\frac{n-1}{n}, n=2,3,...$ see \cite{LyubNes}. 
Note that in these cases the density of the lattice giving a non-frame Gabor system is a most $2$. It is known that for
an arbitrary (Wiener) window $g$ the system $\gG(g;\alpha,\beta)$ 
is a frame for $L^2(\mathbb{R})$ if the density of lattice $\alpha\mathbb{Z}\times\beta\mathbb{Z}$ exceeds some critical one
(depending on $g$ and $\beta$ of course). A lot of efforts have been spent in order to determine this critical density, see e.g. \cite{BC,Daub,RonShen}.
Our approach allows us to construct non-frame rational Gabor systems $\gG(g;\alpha,\beta)$ with lattices of arbitrary large density (see Proposition \ref{nfprop}). This situation becomes different if we restrict the number of summands in \eqref{eq:0_1}.

\begin{theorem}
Let $g$ be of the form \eqref{eq:0_1}, $\Re w_k\neq\Re w_l$, $k\neq l$. Then
$$\biggl{\{}(\alpha,\beta): \alpha\beta\leq\frac{1}{N}\biggr{\}}\subset\mathcal{F}_g.$$
\label{sdth}
\end{theorem}
This therorem is almost precise: we will see that there exists a window $g$ of the form \eqref{eq:0_1} and $\alpha,\beta>0$
with $\alpha\beta=\frac{1}{N-1}$ (and $\Re w_k\neq\Re w_l$, $k\neq l$) such that the corresponding Gabor system $\gG(g;\alpha,\beta)$ does not constitute a frame in $L^2(\mathbb{R})$ (see Proposition \ref{nfprop}). Thus,
 for the rational window one can explicitly find the lower bound for the density which guarantees 
the frame property of the corresponding Gabor system. This gives a partial answer to a question formulated in \cite{Daub}.

{
\subsection{Concluding Remarks}

\subsubsection*{Infinite number of poles} We are able to generalize Theorem \ref{posth} to class of Herglotz functions
with infinite number of poles.
\begin{theorem}
Let $w_k > 0$ be an increasing separated sequence, $w_{k+1}-w_k \ge 1$ and also $0 < a_k < 2^{-2^{2^{w_k}}}$. Then for all $\alpha, \beta$ with $0 < \alpha \beta \le 1$ the system generated by the function
$$g(x) = \sum_{k = 1}^\infty \frac{a_k}{x-iw_k}$$ is a frame.
\label{infth}
\end{theorem}
The proof is based on Theorem \ref{posth} and some perturbative arguments.
The detailed proofs will appear elsewhere. 

\subsubsection*{Two kernels} Using our approach we can describe the frame set $\mathcal{F}_g$ for all functions $g(x)=\frac{a_1}{x-iw_1}+\frac{a_2}{x-iw_2}$, $a_1,a_2,w_1,w_2\in\mathbb{C}$.
In particular, for $w_1,w_2\in\mathbb{R}$ we have $\Pi\subset\mathcal{F}_g$. The detailed proofs will appear elsewhere. 
}

\subsubsection*{Completeness} 
In contrast to the frame property we always have the completeness of rational Gabor systems.
\begin{theorem}
Let function $g(x) = \sum \frac{a_k}{x-iw_k}$ be such that $\Re w_k\ne \Re w_l, k\ne l$. Then the system $\gG(g;\alpha,\beta)$ is complete in $L^2(\R)$ if and only if $\alpha \beta \leq 1$.
\label{compl}
\end{theorem}

\subsubsection*{Multiple poles} The right hand-side of \eqref{eq:0_1} is the general form of a rational function in $L^2(\mathbb{R})$ without multiple poles.
Our approach is applicable for rational functions with multiple poles. However, to avoid inessential technicalities we prefer to consider only rational functions
with simple poles.

\subsection{The structure of the paper} The article is organized as follows. In Section \ref{FC} we give necessary and sufficient conditions for
rational Gabor system to be a frame. This characterization will be used in Sections \ref{Spos}, \ref{Sirr},  \ref{or}.
  In Section \ref{Spos} we prove Theorem \ref{posth} and highlight connections to dynamical systems. In Section \ref{Sirr} we prove Theorems  \ref{irrational} and \ref{irrth}. Finally, in Section \ref{or} we prove Theorem Theorem \ref{close1th}, Theorem \ref{close2th}, Theorem \ref{sdth}, and construct counterexamples.
 In Section \ref{shiftinv} we discuss the connections with theory of shift-invariant subspaces. 

\smallskip

Throughout this paper, $U(z)\lesssim V(x)$ (equivalently $V (z) \gtrsim U(z)$) means that there exists a constant $C$
such that $U(z) \leq CV (z)$ holds for all $z$ in the set in question, which may be a Hilbert
space, a set of complex numbers, or a suitable index set. We write $U(z) \asymp V (z)$ if both
$U(z)\lesssim V (z)$ and $V (z)\lesssim U(z)$.

\section{Frame Criterion\label{FC}}

In this section we give necessary and sufficient conditions for an arbitrary rational function $g\in L^2(\R)$ to generate a frame for given $\alpha,\beta$. This is the key step in the proofs of Theorems \ref{posth}-\ref{irrth}. 

Let 
\beq
\label{A01}
g(t)=\sum_{k=1}^N \frac {a_k}{t-iw_k}, \ a_k, w_k \in \CC, \ a_k \neq 0, \  \Re w_k \neq 0.
\eeq

This is the general form of a rational function in $L^2(\mathbb{R})$ without multiple poles. We add the factor $i$ in the denominator for the sake of convenience. 

\subsection{Multipliers $m_s$ and the main criterion}

Given $\alpha, \beta >0$, $\alpha \beta \leq 1$,
we study the frame property in $L^2(\R) $ of the Gabor  system  
\beq
\label{A02}
\gG(g;\alpha,\beta)=\{g_{m,n}(t)\}_{m,n\in \Z}, \ g_{m,n}(t)=e^{2\pi i\beta t n} g(t-\alpha m).
 \eeq
   
 It is immediate that the system $\gG(g; \alpha, \beta)$ is a frame if and only if the system $\gG(g_\beta; \alpha\beta, 1)$ is a frame, $g_\beta(t) = g(t/\beta)$. Since $g_\beta$ is also a rational function it sufficient to consider only the case $\beta = 1$, $\alpha \in (0, 1]$ which we will assume from now on.
 
  For $k=1, \ \ldots , \ N$ and $s=0,\ \ldots, \ N-1$ denote
   \beq
 \label{A07}
 A_{k,s} =(-1)^s \sum_{j_1<j_2,...,< j_s, j_l\neq k}
         e^{\frac {2\pi} \alpha (w_{j_1} + \ldots + w_{j_s})},
\eeq   
the  sum is taken over pairwise different $j_l$'s such that $j_l\neq k$. Put
\beq 
\label{A15}
m_s(\xi)= \sum_{k=1}^N a_k A_{k,s} e^{2\pi \xi w_k}.
\eeq

\begin{Thm}
 The following statements are equivalent:
 \begin{enumerate}
 \item $\gG(g;\alpha,1)$ is a frame in $L^2(\R)$,
 \item
  \beq
\label{A19a} 
\int_0^{\frac 1 \alpha }\sum_{n\in \Z} \left | 
                          \sum_{s=0}^{N-1} G\left(\xi+n+\frac s \alpha\right) m_s(\xi) \right |^2 d\xi \asymp ||G||^2_{L^2(\R)}, \   G\in L^2(\R) .
\eeq
 \end{enumerate}
 \label{mainth}
\end{Thm}

\subsection{Proof of Theorem \ref{mainth}. Step 1} 

In order to establish the frame property of \eqref{A02} we need to prove
 \beq
 \label{A03} 
 \sum_{m,n} |\langle g_{m,n}, f\rangle|^2 \asymp \| f\| ^2, \quad f\in L^2(\R).
 \eeq

We have 
 \beq
 \label{A04}
\left (  \sum_{m,n} |\langle g_{m,n}, f\rangle|^2  \right )^{\frac 12} =
\sup \left \{\left | \sum_{m,n} c_{m,n}\langle g_{m,n}, f\rangle \right |;  
             \ \sum_{m,n}|c_{m,n}|^2   \leq 1  \right \} .
 \eeq
 
\subsection{ Step 2. \label{st2}}
 Given  ${\bf c}=\{c_{m,n}\} \in l^2(\Z\times\Z)$ we fix $n$ and consider
 \beq
\label{A05}
S_n=  \sum_{m} c_{m,n}\langle g_{m,n}, f\rangle = 
\int_{-\infty}^\infty \overline{f(t)} e^{2\pi i nt} \sum_{k=1}^N a_k 
             \sum_m\frac{c_{m,n}}{t-(\alpha m + i w_k)} dt.
\eeq

Denote
 \begin{equation}
  p_j(t)= 1- e^{i\frac {2\pi}\alpha (t-iw_j)}; \ P(t)=\prod_{j=1}^N p_j(t).  
 \end{equation}
 
 We represent $P(t)p_j(t)^{-1}$ as a trigonometric polynomial
 \beq
 \label{A06}
 P_k(t) =P(t)p_k(t)^{-1}= \sum_{s=0}^{N-1} A_{k,s} e^{i \frac{2\pi}\alpha st}.
 \eeq

 Consider the entire function
 \beq
 \label{A08}
 h_n(t)= \left ( 1 - e^{i \frac{2\pi}\alpha t}  \right ) \sum_m 
        \frac {c_{m,n}} {t- \alpha m}.
\eeq        
 We have 
 \beq
 \label{A09}
 \sum_m 
        \frac {c_{m,n}} {t- (\alpha m+iw_k)} = \frac{P_k(t)}{P(t)} h_n(t-i w_k),
\eeq
  respectively
 \beq 
 \label{A10}        
 S_n = \int_{-\infty}^\infty \frac{\overline{f(t)}}{P(t)} e^{2\pi i n t}
       \sum_{k=1}^N a_kP_k(t)h_n(t-iw_k)dt.
\eeq       

By combining the classical sampling and the Paley-Wiener theorems we have
\beq
\label{A11}
h_n(t)= \int_0^{\frac 1 \alpha} e^{2 \pi  i t \xi} \check h_{n}(\xi) d\xi,
\eeq
and 
\beq
\label{A12}
\check h_n\in L^2(0, 1/\alpha);  \ \| \check h_n\|_{ L^2(0, 1/\alpha)}\asymp
 \|\{c_{m,n}\}_m\|_{l^2(\Z)}.
 \eeq  

Let $g(t)=f(t)/\overline{P(t)}$. Since $ |P(t)| \asymp 1$, $ t\in \R$,  we have $\|f\|\asymp \|g\|$  and also
\beq
\label{A13}
  S_n=\int_{-\infty}^\infty \overline{g(t)} e^{2\pi i n t}
             \sum_{s=0}^{N-1}e^{i \frac {2\pi} \alpha st} M_s(t) dt,
\eeq
where              
 \beq
 \label{A14}
 M_s(t)= \sum_{k=1}^N a_k A_{k,s} h_n(t-iw_k) =
      \int_0^{\frac 1 \alpha}e^{2 \pi i \xi t} \check h_n (\xi)
            \sum_{k=1}^N a_k A_{k,s} e^{2\pi \xi w_k}  d\xi.
\eeq                 

\subsection{Step 3. \label{st3}}      

The Parseval's identity now yields
\beq
\label{A16}
S_n=\int_0^{\frac 1 \alpha} 
             \left [  
\sum_{s=0}^{N-1} \overline{G\left(\xi+n+\frac s \alpha\right)}  m_s(\xi) 
              \right ] \check h _n(\xi) d\xi,
\eeq
 { where
 $G$ is the Fourier transform of $ g$ which satisfies
$
\|G\|_{L^2} =\|g\|_{L^2} \asymp \|f\|_{L^2} .
$
  } 

Finally,
\beq
\label{A17}
\sum_{m,n} |\langle g_{m,n}, f\rangle|^2 \asymp
                \sum_n \left [  \sum_{s=0}^{N-1} \overline{G\left(\xi+n+\frac s \alpha\right)}  m_s(\xi) 
              \right ] \check h _n(\xi) d\xi.
\eeq

\subsection{Step 4.} Observe that the sequence $\{\check h _n\}_{n\in \Z}$ runs through the   whole $l^2(\Z, L^2(0,1/\alpha))$  as $\{c_{m,n}\}$ runs through the whole
$l^2(\Z\times \Z)$ and also  $\|\{\check h _n\}_{n\in \Z}\|_{l^2(\Z, L^2(0,1/\alpha))}\asymp \|\{c_{m,n}\}\|_{l^2(\Z\times \Z)}$.

Set 
\beq
\label {A18}
 \check h_n(\xi) =   \sum_{s=0}^{N-1} G\left(\xi+n+\frac s \alpha\right) \overline {m_s(\xi)}.  
\eeq

We have $\|\{\check h _n\}_{n\in \Z}\|_{l^2(\Z, L^2(0,1/\alpha))}\asymp \|G\|_{L^2(\R)} \asymp\|f\|_{L^2(\R)}$ and, hence, the system $\gG(g;\alpha,1)$  is a frame if and only if
\beq
\label{A19} 
\int_0^{\frac 1 \alpha }\sum_{n\in \Z} \left | 
                          \sum_{s=0}^{N-1} G\left (\xi+n+\frac s \alpha \right ) m_s(\xi) \right |^2 d\xi \asymp 1, \   G\in L^2(\R), \ \|G\| \asymp 1.
\eeq
We arrive  to \eqref{A19a}.                     
$ \square$
    
\begin{remark}Since all functions $m_s(\xi)$ are bounded, the upper estimate in $(ii)$ is always true.
\end{remark}

\section{Frame property for Herglotz functions \label{Spos}}

In this section we prove Theorem \ref{posth}.


\subsection{Frobenius matrices} The proof of Theorem \ref{posth} is based on the Lemma \ref{Frobenius lemma} about Frobenius matrices which was communicated to us by Ivan Bochkov. First we  recall the definition of Frobenius matrix.

\begin{definition}
Let $p(z) = z^n + b_{n-1}z^{n-1} + \ldots + b_0$ be a unitary polynomial.  The Frobenius matrix associated with $p$ is the following matrix
\begin{equation}F(p) = 
\begin{pmatrix}
-b_{n-1} & -b_{n-2} & \ldots &  -b_0\\
1 & 0 & \ldots & 0\\
0&1&\ldots&0\\
\vdots&\vdots&\vdots&\vdots\\
0&0&\ldots&0
\end{pmatrix}.
\end{equation}
\end{definition}
We refer the reader to \cite{St} and \cite{MaM} for the detailed presentation
of properties of such matrices. In particular, $p$ is the characteristic polynomial of  $F(p)$.  The next lemma is the key step in the proof of Theorem \ref{posth}, we also think that it is of independent interest.

\begin{lemma}\label{Frobenius lemma}
Let the sequence $\{\mu_k\}_{k=1}^{n+1}$ be such that $1 > \mu_1 > \mu_2 > \ldots > \mu_{n + 1} > 0$. Consider the set $\mathcal{P}$ of all polynomials $p(x) = (x-\lambda_1)(x-\lambda_2)\ldots (x-\lambda_n)$ such that their zeroes interlace with $\mu_k$'s, i.e. $\mu_k \ge \lambda_k \ge \mu_{k+1}$. Then there exist constants $ C > 0, 0 < c < 1$ depending only on the numbers $\mu_k$ such that   for any polynomials $p_1, \ldots, p_m \in \mathcal{P}$ we have  
\begin{equation}
||F(p_1)F(p_2)\ldots F(p_m)|| \le Cc^m.
\end{equation}
\end{lemma}
We do not specify matrix norm here since all norms in finite-dimensional space are equivalent.
We postpone the proof and first obtain Theorem \ref{posth} from Lemma \ref{Frobenius lemma}.

\subsection{Proof of Theorem \ref{posth} Step 1\label{s1p}}
As before we assume $\beta=1$, $\alpha\leq1$ and prove the relation \eqref{A19a}.
We truncate the integral in \eqref{A19a} and prove the stronger estimate

\begin{equation}\label{frame inequality one }
 \intt_0^1\summ_{l\in\mathbb{Z}} \left|\summ_{s = 0}^{N-1} G\left(\xi + l + \frac{s}{\alpha}\right)m_s(\xi)\right|^2d\xi
\gtrsim ||G||_2^2,\quad G\in L^2(\mathbb{R}).
\end{equation}
We remind that the functions $m_s(\xi)$ are determined in \eqref{A15}.



Put $$n = N - 1, \, \mu_k = e^{-2\pi w_k/\alpha}, k=1,2,\ldots n+1.$$  
We have $1 > \mu_1 > \mu_2 > \ldots > \mu_{n+1} > 0$. 
It follows from the positivity of $a_k$'s  that the roots of the polynomial $$p_\xi(z) = \frac{m_0(\xi) + m_1(\xi)z + \ldots + m_n(\xi) z^n}{m_n(\xi)}$$ satisfy the assumptions of the Lemma \ref{Frobenius lemma}.

The estimate \eqref {frame inequality one }  is now equivalent to 

\begin{equation}\label{bound}
 ||LG||_2 \gtrsim||G||_2, \, G\in L^2(\rl),
 \end{equation}
 where the operator $L:L^2(\R) \to L^2(\R)$ is given by the formula 
 $$
 (LG)(\xi) = \summ_{s = 0}^n G\left(\xi + \frac{s}{\alpha}\right)m_s(\{ \xi\}).$$ 
 Here $\{\xi\}$ denotes the fractional part of $\xi$.

\subsection{Step 2} From the definition of $m_s$ we have  
\begin{equation}
m_n(\{\xi\}) = (-1)^{n}e^{\frac{2\pi}{\alpha}(w_1 + \ldots + w_N)}\sum_{k = 1}^N a_ke^{2\pi \{\xi\} w_k - \frac{2\pi}{\alpha}w_k}.
\end{equation}
The absolute value of $m_n$ is bounded from above and from below by some positive constants. 
Without loss of generality we can consider instead of $LG$ the operator $PG$ 
$$(PG)(\xi) = \frac{(LG)(\xi)}{m_n(\{ \xi\})}.$$


It suffices to construct left inverse for $P$, i.e. solve
$$PG = h,\, h\in L^2(\R).$$
We rewrite this equation in the form
\begin{equation}\label{our}
G\left(\xi + \frac{n}{\alpha}\right) = h(\xi) - \summ_{s = 0}^{n-1}\frac{m_s(\{ \xi\})}{m_n(\{\xi\})}G\left(\xi + \frac{s}{\alpha}\right).
\end{equation}

\subsection{Step 3} We transform equation \eqref{our} to a dynamical system.
Consider  the vector-functions $\Gamma,H \in L^2(\R, \CC^n)$: 
 
\begin{equation*}
\Gamma(\xi) = \begin{pmatrix}
G(\xi + \frac{n-1}{\alpha})\\
G(\xi + \frac{n-2}{\alpha})\\
\vdots\\
G(\xi)
\end{pmatrix}, \qquad
H(\xi) = \begin{pmatrix}
h(\xi)\\
0\\
\vdots\\
0
\end{pmatrix}.
\end{equation*}

 In this notation equation \eqref{our} can be rewritten as $$\Gamma\left(\xi + \frac{1}{\alpha}\right) = F\left(p_{\{ \xi\}}\right)\Gamma(\xi) + H(\xi).$$ Iterating this formula we get 

\begin{equation}
\Gamma\left(\xi + \frac{1}{\alpha}\right) = F\left(p_{\{ \xi\}}\right)F\left(p_{\{ \xi - \frac{1}{\alpha}\}}\right)\Gamma\left(\xi - \frac{1}{\alpha}\right) + F\left(p_{\{ \xi\}}\right)H\left(\xi - \frac{1}{\alpha}\right) + H(\xi).
\end{equation}

and

$$
\Gamma\left(\xi + \frac{1}{\alpha}\right) = F\left(p_{\{ \xi\}}\right)F\left(p_{\{ \xi - \frac{1}{\alpha}\}}\right)\ldots F\left(p_{\{ \xi - \frac{k}{\alpha}\}}\right)\Gamma\left(\xi - \frac{k}{\alpha}\right) +$$ $$ \summ_{s = 0}^k F\left(p_{\{ \xi\}}\right)F\left(p_{\{ \xi - \frac{1}{\alpha}\}}\right)\ldots F\left(p_{\{ \xi - \frac{s-1}{\alpha}\}}\right)H\left(\xi - \frac{s}{\alpha}\right).
$$

By Lemma \ref{Frobenius lemma} the coefficients in front of $\Gamma(\xi - \frac{k}{\alpha})$ and $H(\xi - \frac{s}{\alpha})$ decay exponentially. Therefore, we can pass to the limit and get 
$$\Gamma\left(\xi + \frac{1}{\alpha}\right) =  \summ_{s = 0}^\infty F\left(p_{\{ \xi\}}\right)F\left(p_{\{ \xi - \frac{1}{\alpha}\}}\right)\ldots F\left(p_{\{ \xi - \frac{s-1}{\alpha}\}}\right)H\left(\xi - \frac{s}{\alpha}\right).$$
Thus, we get $||\Gamma||_2 \le C||H||_2 = ||h||_2$. Hence,  $||G||_2 \le C'||h||_2 = C'||PG||_2$. That is, $||PG||_2 \gtrsim ||G||_2$ and, subsequently, $||LG||_2 \gtrsim ||G||_2$  which is the desired estimate \eqref{bound}.
$\square$


\subsection{Proof of Lemma \ref{Frobenius lemma}. Preliminaries}

We are going to construct a norm $||v||_{*}$ on $\R^n$ such that for all $p\in \mathcal{P}$ we have $||F(p)||_*\le c < 1$. This implies by induction that $||F(p_1)\ldots F(p_m)||_* \le c^m$. Since all norms on the finite-dimensional space are equivalent we get the result. 

\medskip

We will actually construct a norm in which $F(p)^T$ is contractive uniformly for all $p\in \mathcal{P}$. If we are able to do so, then matrices $F(p)$ will be contractive in the dual norm.

\medskip

The proof consists of two steps. In the first step we show that it is enough to consider only polynomials
 $p$ such that  $\lambda_k$'s is either $\mu_k$ or $\mu_{k+1}$ for all $k$. Moreover, we will show that among them we can actually study only those for which all $\lambda_k$ are distinct, that is polynomials  $p_l(x) = \prod_{k\ne l}(x-\mu_k)$, $l = 1, \ldots, n+1$. In the second step we will construct a norm in which all matrices $F(p_l)^T$ are uniformly contractive. 

\medskip

\subsection{ Step 1. Reduction to $n+1$ matrices.} 
Let $p(x) = \prod_{k = 1}^n (x-\lambda_k)$ be an arbitrary polynomial in $\mathcal{P}$. Assume that  $\lambda_k\ne \mu_k, \mu_{k+1}$ for some $k$. Since $\mu_k > \lambda_k > \mu_{k+1}$ we can find positive numbers $a, b\in \R$ with $a+b = 1$ such that $(x-\lambda_k) = a(x-\mu_k) + b(x-\mu_{k+1})$. Denoting $q(x) = \frac{x-\mu_k}{x-\lambda_k}p(x)$ and $r(x) =\frac{x-\mu_{k+1}}{x-\lambda_k}p(x)$ we get $p(x) = aq(x)+br(x)$ and therefore
 $F(p)^T = aF(q)^T + bF(r)^T$.  Repeating this procedure with $q$ and $r$ we can express $F(p)^T$ as a convex combination of matrices of the form $F( \rho)^T$ where $\rho\in \mathcal{P}$ and all its roots are from 
$\{ \mu_1, \ldots , \mu_{n+1}\}$. By the triangle inequality if $F(\rho)^T$ are contractive for all such polynomials $\rho$ then $F(p)^T$ is contractive as well.

\medskip

Now we show that it is enough to consider only the polynomials $p_l$. Let us consider all $2^n$ polynomials $p\in \mathcal{P}$ with $\lambda_k = \mu_k$ or $\lambda_k = \mu_{k+1}$ and denote by $K$ the convex hull of the corresponding matrices $F(p)^T$. Since it is a convex hull of finitely many points it is a polytope. It is well-known that any polytope is a convex hull of its vertices which are exactly the points that can not be written as a convex combination of  other points from this polytope. 

Let $p(x) = \prod_{k=1}^n (x-\lambda_k)$ be such that $\lambda_k = \mu_k$ or $\lambda_k = \mu_{k+1}$ for all $k$ and moreover $\lambda_l = \lambda_{l+1}=\mu_{l+1}$ for some $l$. There exist positive numbers $a, b$ such that $a+b = 1$ and $(x-\mu_{l+1}) = a(x-\mu_l) + b(x-\mu_{l+2})$. Denoting $q(x) = \frac{x-\mu_l}{x-\mu_{l+1}}p(x)$ and $r(x) =\frac{x-\mu_{l+2}}{x-\mu_{l+1}}p(x)$ we get $p(x) = aq(x)+br(x)$ and therefore $F(p)^T = aF(q)^T + bF(r)^T$. On the other hand we have $F(q)^T, F(r)^T\in K$ and both of them are not equal to $F(p)^T$. That is, we decomposed $F(p)^T$ as a convex combination of other points from $K$. Thus, it is not a vertex of $K$.

Therefore the only possible candidates for the vertices of $K$ corresponds to polynomials with distinct roots, that is $p_l$'s. That is, all points from $K$ are convex combinations of $F(p_l)^T$'s and so if $F(p_l)^T$'s are contractive for all $l$  then all  other matrices from $K$ are contractive as well. 


\begin{remark}  \ Instead of appealing to the theory of polytopes one can more carefully decompose $p$ into the sum of two other  
polynomials such that at each step the number of repeated roots decreases and then continue the process until there are none.
\end{remark}

\subsection{Step 2. Construction of the contractive norm for $F(p_l)^T$'s.}
 

Let us  identify $\R^n$ with the space   $\mathcal {P}_{n-1}$ of all  polynomials of degree less than $n$ in a way that 
$\alpha= (\alpha_{n-1}, \ \ldots \ , \alpha_0)^T\in \R^n$ corresponds to the polynomial
$q_\alpha(z)=\alpha_{n-1}z^{n-1} + \ldots + \alpha_0\in \mathcal {P}_{n-1}$. One can see that the action of $F(p)^T$ on the polynomial 
$q\in \mathcal {P}_{n-1}$ 
corresponds to the operation $${q(z)\mapsto zq(z)({\rm mod}  \ p(z))}.$$

For each $l=1,2, \ \ldots \ , n+1$ consider the linear functional which sends the  polynomial $q\in \mathcal {P}_{n-1}$ to $q(\mu_l)$. 
 Since these are $n+1$ linear functionals on the $n$-dimensional vector space there is a linear dependence between them: 
 $$a_1q(\mu_1) + a_2q(\mu_2) + \ldots + a_{n+1}q(\mu_{n+1}) = 0$$ for all polynomials $q\in \mathcal {P}_{n-1}$. 
  Moreover, since the values of $q$ at any $n$ different points uniquely determine $q$ none of $a_k$'s vanishes. 
Put
$$
||q||_\mu = |a_1q(\mu_1)| + \ldots + |a_{n+1}q(\mu_{n+1})|. 
$$ 
 Since $q$ is uniquely determined by $q(\mu_1), \ldots , q(\mu_{n+1})$, this is a norm on $ \mathcal {P}_{n-1}$. We show that $||F(p_l)^T q||_\mu \le \mu_1 ||q||_\mu$ for all $l$. Since $\mu_1 < 1$ this implies the result.
 
 \medskip

Put $r(x) = (F(p_l)^Tq)(x)$. We have 
$$r(\mu_k) = \mu_kq(\mu_k) \text{ for all } k\ne l$$ and $$a_lr(\mu_l)=-\sum_{k\neq l}a_kr(\mu_k).$$ Therefore

\begin{multline}\label{new norm}
||r||_\mu = \summ_{k\ne l} |\mu_k a_kq(\mu_k)| + |a_lr(\mu_l)| = \summ_{k\ne l} |\mu_ka_kq(\mu_k)| + |\summ_{k\ne l}a_k\mu_kq(\mu_k)|=\\
          \mu_1 \left (  \summ_{k\ne l} \left|\frac {\mu_k}{\mu_1 }a_kq(\mu_k)\right| + \left|\summ_{k\ne l}a_k\frac{\mu_k}{\mu_1}q(\mu_k)\right| \right ).
\end{multline}
For $s_k\in [0,1]$, $x_k\in \R, k\ne l$ we have
$$\left|\sum_{k\ne l} s_k x_k\right|-\left|\sum_{k\ne l} x_k\right| \leq \left|\sum_{k\ne l}x_k(1-s_k)\right|\leq \sum_{k\ne l}|x_k|(1-s_k). $$
Hence,
$$
\sum_{k\ne l} |s_k x_k| + \left|\sum_{k\ne l} s_k x_k \right| \leq \sum_{k\ne l} |x_l| + \left|\sum_{k\ne l}  x_k \right|.
$$
Setting $x_k = a_kq(\mu_k)$, $s_k = \frac{\mu_k}{\mu_1}$ we get 
\begin{equation}\label{old norm}
||r||_\mu \le \mu_1 ||q||_\mu,
\end{equation}
as required. 
$\square$


\begin{remark}
It is easy to see from the proof that we can choose $c = \mu_1$. But if two $\mu$'s approach each other then constant $C$ may blow up. But one can check carefully analysing the proof of the above lemma that for say $c = \frac{1+\mu_1}{2}$ we can choose $C$ depending only on $\mu_1$ and $n$.
\end{remark}

\section{Irrational densities \label{Sirr}}

In this section we prove Theorem \ref{irrational} and Theorem \ref{irrth}. The main ingridient of the  proofs is careful analysis 
of the rank of the finite matrices with the rows $(0,..., m_0(\xi),....,m_{N-1}(\xi),0,..,0)$. 

%

\subsection{Preliminaries} Without loss of generality we may assume, as before, $\beta=1$, $\alpha\in(0,1)\setminus\mathbb{Q}$
because the rescaling $g\mapsto g_\beta(t):=g(t\slash\beta)$ as in Section \ref{Spos} leads one just to rescaling of the corresponding functions $m_0(\xi)$, $\hat{g}(\xi)$.


\medskip

Put
$$M(\xi)=(m_0(\xi),m_1(\xi),...,m_{N-1}(\xi)).$$
For  the reader's convenience we write $\mathbb{O}_k$ for zero row of length $k$; when the length is clear from the context we suppress the subscript $k$.

We assume $\alpha>1\slash2$. The (easier) case $\alpha<1\slash2$ can be done similarly.

\medskip

Let $\tau=\frac{1}{\alpha}-1$. For any fixed $\xi\in(1,\frac{1}{\alpha})$ consider the sequence
$$\{\xi,\xi-1,\xi-1+\tau,\xi-1+2\tau,....,\xi-1+k_1\tau\},$$
here $k_1\in\mathbb{N}$ is the first number such that $\xi-1+k_1\tau\in(1,1\slash\alpha)$. We can repeat the procedure
starting from $\xi-1+k_1\tau$ and take the first $k_2$ such that $\xi-2+(k_1+k_2)\tau\in(1,1\slash\alpha)$, and so on. After $l$ steps we obtain the sequence
$$S_{\xi,l}:=\{\xi,\xi-1,\xi-1+\tau,\xi-1+2\tau,....,\xi-1+k_1\tau, \xi-2+k_1\tau,$$ $$\xi-2+(k_1+1)\tau...,\xi-2+(k_1+k_2)\tau),$$
$$...,$$
$$\xi-l+K\tau, \xi-l-1+K\tau\},$$
where $K=k_1+k_2+...+k_l$.
 With any such sequence $S(\xi,l)$ we associate finite $(K+l+1)\times(K+N)$ matrix $D=D(\xi,l)$,
$$D=D(\xi,l)=\begin{pmatrix} M(\xi) & \mathbb{O} \\
M(\xi-1) & \mathbb{O} \\
\mathbb{O}_1 & M(\xi-1+\tau) & \mathbb{O} \\
\mathbb{O}_2 & M(\xi-1+2\tau) & \mathbb{O} \\
& \ldots\\
\mathbb{O}_{k_1-1} & M(\xi-1+(k_1-1)\tau) & \mathbb{O} \\
\mathbb{O}_{k_1} & M(\xi-1+k_1\tau) & \mathbb{O} \\
\mathbb{O}_{k_1} & M(\xi-2+k_1\tau) & \mathbb{O} \\
\mathbb{O}_{k_1+1} & M(\xi-2+(k_1+1)\tau) & \mathbb{O} \\
& \ldots\\
\mathbb{O}_{K} & M(\xi-l+K\tau )  \\
\mathbb{O}_{K} & M(\xi-l-1+K\tau)  \\
\end{pmatrix}.$$
We put attention of the reader  to the (a bit) non-traditional form of representation
of this matrix: the each "column" $ $ consists of strings of various length. In the next section 
we will see how does this matrix appear and also explain its structure in more details.

The next lemma is the key technical step in the proof of Theorem \ref{irrational}.

\begin{lemma}
\label{tech}There exist $\hat{\xi}\in (1,\frac{1}{\alpha})$, $l\in \N$, $l>N$, and $\delta>0$ such that for any $\xi\in[\hat{\xi}-\delta,\hat{\xi}+\delta]$ rank of the matrix $D(\xi,l)$ is $K+N$.

\end{lemma}

We postpone the proof of this lemma and first deduce Theorem \ref{irrational} from Lemma \ref{tech}.

\subsection{Proof of Theorem \ref{irrth}\label{irrpr}}
We have to establish relation \eqref{A19a}.
The $\lesssim$ part \eqref{A19a} is straightforward since all $m_s(\xi)$ are bounded.
In order to prove opposite inequality it suffices to construct the left inverse $\mathcal{L}^{-1}$
to the operator $\mathcal{L}:L^2(\mathbb{R})\rightarrow\ell^2(L^2(0,1\slash\alpha))$
defined by the relation
\begin{equation}
\mathcal{L}:G\mapsto \biggl{\{}\sum_{s=0}^{N-1}G\left (\xi+n+\frac{s}{\alpha}\right )m_s(\xi)\biggr{\}}_n. \label{loper}
\end{equation}
As before, we restrict ourselves to the (more complicated) case $\alpha>1\slash2$ and denote $\tau=\frac{1}{\alpha}-1$.

Given $\gamma=\{\gamma_n(\xi)\}_{n\in\mathbb{Z}}\in\Im\mathcal{L}$
we have to solve the infinite sequence of equations with respect to $\{G(\xi+n+s\slash\alpha)\}_n$
\begin{equation}
\sum_{s=0}^{N-1}G\biggl{(}\xi+n+\frac{s}{\alpha}\biggr{)}m_s(\xi)=\gamma_n(\xi),\qquad 
n\in\mathbb{Z}, \quad \xi\in\biggl{(}0,\frac{1}{\alpha}\biggr{)}.
\label{GMeq}
\end{equation}

We use notations from the previous section. For $\xi\in(1,1\slash\alpha)$ we will try to choose 
a subsystem of \eqref{GMeq} which can be resolved with respect to variables 
\begin{equation}
\biggl{\{}G\biggl{(}\xi+\frac{j}{\alpha}\biggr{)}\biggr{\}}_{j\in\mathbb{Z}}.
\label{zalpha}
\end{equation}

This leads us to the matrix $D(\xi,l)$. Indeed, we have $\xi\in(1,1\slash\alpha)$, $\xi-1\in(0,\tau)$.
The two equations in \eqref{GMeq} written for $\xi$ with $n=0$ and for $\xi-1$ with $n=1$  contain
{\it the same } selection of variables $G(\xi+s\slash\alpha)$, $s=0,1,...,N-1$.
The coefficients in these equations belong to the strings $M(\xi)$, $M(\xi-1)$. 
We complete these strings by the corresponding number of zeroes and obtain 
the first two rows of the matrix $D(\xi,l)$. The equation in \eqref{GMeq} with $\xi-1+\tau$ and
$n=2$ has the form
$$\sum_{s=0}^{N-1}G\biggl{(}\xi+\frac{s+1}{\alpha}\biggr{)}m_s(\xi-1+\tau)=\gamma_2(\xi-1+\tau).$$
This equation contains the variables $\{G(\xi+s\slash\alpha)\}_{s=1}^N$, its coefficients are the elements of the string
$M(\xi-1+\tau)$. Completing this string by one zero on the left and by the corresponding amount of zeroes on right
we obtain the third row in $D(\xi,l)$.

Repeating this procedure as described in previous section we obtain the whole matrix $D(\xi,l)$.
We remark that the number of unknowns increases by one when we add the equation related
to the shift of the argument by $\tau$ and remains the same, when we add the equation related
to the shift of the argument by $-1$. This will allow us to extract subsystem of \eqref{GMeq} which
contains the same amount of equations and variables to be determined.

Moreover, for each $\xi\in(1,1\slash\alpha)$ we can explicitly write the matrix of the operator
applied to the sequence $\{G(\xi+j\slash\alpha)\}_{j\in\mathbb{Z}}$. This matrix consist of
single and double strings $M(\cdot)$ shifted with respect to each other.


\begin{equation}
L_\xi = \begin{pmatrix}
& \ldots\\
\mathbb{O}_{-k_{-1}} & M(\xi+1-k_{-1}\tau) & \mathbb{O}\\
\mathbb{O}_{-k_{-1}} & M(\xi-k_{-1}\tau) & \mathbb{O}\\
& \ldots\\
\mathbb{O}_{-2} & M(\xi-2\tau) & \mathbb{O}\\
\mathbb{O}_{-1} & M(\xi-\tau) & \mathbb{O}\\
\mathbb{O}_0 & M(\xi) & \mathbb{O} \\
\mathbb{O}_0 & M(\xi-1) & \mathbb{O} \\
\mathbb{O}_{+1} & M(\xi-1+\tau) & \mathbb{O} \\
\mathbb{O}_{+2} & M(\xi-1+2\tau) & \mathbb{O} \\
& \ldots\\
\mathbb{O}_{k_1} & M(\xi-1+k_1\tau) & \mathbb{O} \\
\mathbb{O}_{k_1} & M(\xi-2+k_1\tau) & \mathbb{O} \\
& \ldots\\
\mathbb{O}_{k_1+k_2} & M(\xi-2+(k_1+k_2)\tau )  & \mathbb{O}\\
\mathbb{O}_{k_1+k_2} & M(\xi-3+(k_1+k_2)\tau)  & \mathbb{O}\\
& \ldots
\end{pmatrix},
\label{Lmatrix}
\end{equation}

here $\mathbb{O}_{q}$ indicates the shift of the corresponding string $M(\cdot)$ to the left or to the right
depending on the sign of $q$.

\medskip

Observe that the rows of the operator $L_\xi$ have a similar structure. So, we can start with $\xi-1+k_1\tau$ instead of $\xi$. Similarly, if $\xi\notin (1, \frac{1}{\alpha})$ we can first add to it $r\tau$ for some $r\in \N$ so that $\xi + r\tau \in (1, \frac{1}{\alpha})$
and procede from there.

\medskip

Chose $l$, $\hat{\xi}$ and $\delta$ as in Lemma \ref{tech} and, for each 
$t\in\mathbb{Z}$, denote by $\xi_t$ the point of $t$-th return 
of the original point $\xi$ into yne interval $(1,1\slash\alpha)$:
$$\xi_t=\xi-t+k_1+k_2+...+k_t,$$
this is the first argument of $M$ in the $t$-couple of double rows in \eqref{Lmatrix}. In this notation we have $\xi=\xi_0$.
Since $\alpha\not\in\mathbb{Q}$ the set $\xi_t$ is dense in $(1,1\slash\alpha)$.

By shifting of numeration we may assume that $\xi_0\in[\hat{\xi}-\delta,\hat{\xi}+\delta]$, and also we can
choose $t\in\mathbb{N}$ so that the point $\xi_{-t}=\xi+t-(k_{-1}+...+k_{-t})\in[\hat{\xi}-\delta,\hat{\xi}+\delta]$.
Consider the submatrix of $L_\xi$ located between the rows
$(\mathbb{O}, M(\xi+t-(k_{-1}+k_{-2}+...+k_{-t})\tau),\mathbb{O})$ and $(\mathbb{O}, M(\xi-l+(k_{1}+k_{2}+...+k_{l})\tau),\mathbb{O})$:

%
%
$$C=C(\xi,l,t)=\begin{pmatrix}
 M(\xi+t-(k_{-1}+...+k_{-t})\tau) & \mathbb{O}\\
 M(\xi + t - 1-(k_{-1}+k_{-2}+...+k_{-t})\tau) & \mathbb{O}\\
& \ldots\\
\mathbb{O}_{k_{-1}+...+k_{-t}-2} & M(\xi-2\tau) & \mathbb{O}\\
\mathbb{O}_{k_{-1}+...+k_{-t}-1} & M(\xi-\tau) & \mathbb{O}\\
\mathbb{O}_{k_{-1}+...+k_{-t}} &  D(\xi, l)
\end{pmatrix}.$$

This matrix has $\nu:=N+K+(k_{-1}+k_{-2}+...+k_{-t})$ columns.

\subsection{Rank of the matrix $C(\theta,t,l)$}
We will show that $rank(C(\xi,l,t))=\nu$, that is we have to choose collection of $\nu$ rows of $C(\xi,l,t)$
which span the whole space $\mathbb{R}^\nu$. By Lemma \ref{tech} there is a square non-degenerate
submatrix $E(\xi,l)$ of $D(\xi,l)$ of size $(K+N)\times(K+N)$.
We keep the rows which correspond to $E(\xi,l)$ and eliminates the rest of rows $D(\xi,l)$.
Further we eliminate each second row in the couples of double rows, i.e. the rows which
contain the strings $M(\xi+p-(k_{-1}+...+k_{-p}))$, $p=1,...,t$. The remaining rows form $\nu\times\nu$ matrix of the form

$$\begin{pmatrix}
M(\xi + t - (k_{-1}+\ldots + k_{-t})\tau)& \mathbb{O}\\
\mathbb{O}_1&M(\xi + t -1- (k_{-1}+\ldots + k_{-t}-1)\tau)& \mathbb{O}\\
& \ldots &\\
\mathbb{O}_{k_{-1}+...+k_{-t}-1} & M(\xi-\tau) & \mathbb{O}\\
\mathbb{O}_{k_{-1}+...+k_{-t}} &  E(\xi, l)
\end{pmatrix}.$$
This is a block-diagonal matrix
$\begin{pmatrix}X&Y\\0& E(\xi,l)\end{pmatrix}$. 
In addition $X$ is an upper-triangular matrix, its diagonal elements are values $m_0(\xi)$
for some point $\xi\in(\tau,1+\tau)$. They do not vanish and bounded away from zero by the assumption regarding $m_0$.
Thus the matrix $C(\xi,l)$ indeed has full rank.

\subsection{End of the proof of Theorem \ref{irrth}} We can now find $\overrightarrow{G}=\{G(\xi+j\slash\alpha)\}_{j=-k_1-...-k_t}^{K+N}$
which solves the equation of system $\eqref{GMeq}$ which correspond to the selected rows of the matrix $C(\xi,l)$.
The equation of this system which correspond to the rest of the rows of $C(\xi,l)$ will be met automatically
since we assume $\gamma\in\Im L_\xi$. In addition we have
\begin{equation}
\|\overrightarrow{G}\|_2\geq C\|\tilde{\gamma}\|,
\label{GT}
\end{equation}
where $\tilde{\gamma}$ is the section of the sequence $\gamma$ corresponding
the rows of $C(\xi,l)$. The constant $C$ depends on $t$ and the estimate from below
for $|\det E(\xi,l)|$.

Firstly, we observe that number $t$ is uniformly bounded with respect to $\theta$. Indeed, we have an irrational motion with step $\tau$ and 
it's well known that it lands into any given interval in the bounded number of steps regardless of the starting point. 

\medskip

One can choose $\varepsilon>0$ so that, for each $\xi\in[\xi_0-\delta,\xi_0+\delta]$ we have $|\det E(\xi,l)|>\varepsilon$
for the corresponding submatrix $E(\xi,l)$ of $D(\xi,l)$. Therefore the constant $C$ in \eqref{GT} can be chosen uniformly on $\xi\in[\xi_0-\delta,\xi_0+\delta]$.

\medskip

It remains to note that the operator $L_\xi$ can be decomposed into the operators $C(\xi)$ with finite overlapping. 
So, finally we have constructed the bounded left inverse to \eqref{loper}.

\begin{remark}
The proof of Theorem \ref{irrational} can be roughly decomposed into the following ideas: we can use nonvanishing of the function $m_0$ to shift the attention from the number $\xi$ to the number $\xi - \tau \mod \frac{1}{\alpha}$. 
Since $\tau = \frac{1}{\alpha}-1$ and $\frac{1}{\alpha}$ are incommensurable in this way we can come close to any given point on the interval $[0, 1\slash\alpha]$. Thus, it is enough to prove the corresponding bound for any single $\xi_0\in (0, 1\slash\alpha)$ (and its small vicinity), which is done
 in Lemma \ref{tech}, below.
\end{remark}

\subsection{Proof of Lemma \ref{tech}. Step 1} First we observe the identity 
\begin{equation}
\frac{\sum_{s=0}^{N-1}m_s(\xi)z^s}{\prod_{k=1}^N(1-ze^{\frac{2\pi}{\alpha} w_k})}=\sum_{k=1}^N\frac{a_ke^{2\pi\xi w_k}}{1-ze^{\frac{2\pi}{\alpha} w_k}},
\label{Zakm}
\end{equation}
this follows from the definition of the functions $m_s$, see \eqref{A15}. Let
$$u_j=e^{\frac{2\pi}{\alpha}w_j}, j=1,...N.$$
Fix $j\in\{1,...,N\}$ and compare the residue at $z=u^{-1}_j$ of both sides in \eqref{Zakm}.
\begin{equation}
\sum_{s=0}^{N-1}m_s(\xi)u^{-s}_j=a_ju^{1-N}_je^{2\pi\xi w_j}\prod_{l\neq j}(u_j-u_l).
\label{cyclic}
\end{equation}

Assume that the number $l>N$ is already found. The $(K+l+1)\times(K+N)$ matrix $D(\xi,l)$ is composed from $l+1$ "double" $ $ rows
of the form 
$$(\mathbb{O}_{k_1+...+k_s},M(\xi-s+(k_1+...+k_s)\tau), \mathbb{O}),$$  
$$(\mathbb{O}_{k_1+...+k_s},M(\xi-s-1+(k_1+...+k_s)\tau), \mathbb{O})$$
with "single" $ $ rows in between.

In order to transform it to a square matrix it suffices to eliminate $l-N-1$ rows.
Let us eliminate the second rows in the appropriate number of double rows except the first and last ones.
The remaining second rows are of the form 
$$(\mathbb{O}_{k_1+...+k_s},M(\xi-s-1+(k_1+...+k_s)\tau), \mathbb{O})$$
for $s=Q_1,Q_2,...,Q_{N-1}$ for some $0= Q_1<Q_2<...<Q_{N-1}=l$.
Denote the resulting matrix by $F(\xi,l)$  and let
$$q_j=k_{Q_{j-1}}+...+k_{Q_j}$$
be the "distance" $ $ between the rows with numbers $Q_{j-1}$ and $Q_j$.

We will choose the numbers  $Q_1,...,Q_{N-1}$, and also $\hat{\xi}\in(1,1\slash\alpha)$, $\delta>0$,
so that $\det F(\xi,l)\neq0$, $\xi\in(\hat{\xi}-\delta,\hat{\xi}+\delta)$.

\medskip


 We apply \eqref{cyclic} for each $j=1,2,...,N$:

$$F(\xi,l) \begin{pmatrix} u_j^{K+N-1} \\
u_j^{K+N-2} \\...\\
u_j\\
1
\end{pmatrix}=a_je^{2\pi\xi w_j}\prod_{l\neq j}(u_j-u_l)
\begin{pmatrix}
u_j^{K}\\
u_j^{K}e^{-2\pi w_j}\\
u_j^{K-1}e^{-2\pi w_j}e^{2\pi \tau w_j}\\
u_j^{K-2}e^{-2\pi w_j}e^{2\pi 2\tau w_j}\\
\ldots\\
u_j^{K-q_2}e^{-2\pi Q_2w_j}e^{2\pi q_2\tau w_j}\\
u_j^{K-q_2}e^{-2\pi (Q_2+1)w_j}e^{2\pi q_2\tau w_j}\\
\ldots\\
 e^{-2\pi Q_{N-1}w_j}e^{2\pi K\tau w_j}\\
e^{-2\pi (Q_{N-1}+1)w_j}e^{2\pi K\tau w_j}
\end{pmatrix}$$
$$=a_je^{2\pi\xi w_j}\prod_{l\neq j}(u_j - u_l) V_j(l).
$$
For each $j=1,2,...,N$, $V_j(\xi,l)$ is a column of size $K+N$. We observe that
$V_j(l)$ is independent of $\xi$.

Put

$$W(\xi,l):=\begin{pmatrix}
\begin{matrix} u_1^{K+N-1} & u_2^{K+N-1} & \ldots & u_{N}^{K+N-1}\\ 
u_1^{K+N-2} & u_2^{K+N-2} & \ldots & u_N^{K+N-2} \\...\\
u_1 & u_2 & \ldots & u_N\\
1 & 1 & \ldots & 1
\end{matrix}
 & \begin{matrix} \mathbb{O}_{N\times K} \\ \mathbb{I}_K\end{matrix}
\end{pmatrix}
$$
here $\mathbb{O}_{N\times K}$ is the zero $N\times K$ matrix and $\mathbb{I}_{K}$ is the identity matrix of the size $K\times K$. 

For any choice of $Q_1,...,Q_{N-1}$ the determinant $d(\xi)=\det(F(\xi,l)W(\xi,l))$ is an exponential
polynomial, i.e. it is a finite sum of the form
$$d(\xi)=\sum_j\alpha_je^{\xi\beta_j}.$$ 
We are going to prove that for an appropriate choice of  $Q_1,...,Q_{N-1}$, this
polynomial does not vanish identically. This would prove Lemma \ref{tech}. 

$$d(\xi)=e^{2\pi (w_1+...+w_N)\xi}\prod_{j=1}^Na_j\prod_{j\neq l}(u_j-u_l)\det (V_1,..., V_N,F_{N+1},..., F_{N+K}),$$
where $F_j$ is a $j$'th column of the matrix $F(\xi,l)$. We observe that $e^{2\pi (w_1+...+w_N)\xi}\prod_{k=1}^Na_j\prod_{j\neq l}(u_j-u_l)\neq 0,$ which follows from the assumption that $\Re w_j \ne \Re w_l$, $j\ne l$. 

\subsection{Step 2}
It remains to choose $Q_1,...,Q_{N-1}$ so that $\det (V_1,..., V_N,F_{N+1},..., F_{N+K})$ is non-zero. Note that this determinant is also an exponential polynomial in $\xi$. 
Therefore, it suffices to find at least one non-zero coefficient.
We assume that $\Re w_1>\Re w_2>...>\Re w_N$ and we are going to choose $Q_1,...,Q_{N-1}$ so that
the term $e^{2\pi K w_1\xi}$ participates in our polynomial with a non-zero
coefficient.
\medskip

We have
$$M(\xi-1)-e^{-2\pi w_1}M(\xi)=(J_0,J_1,...,J_{N-1}),$$
where $J_s$ does not contain the frequency $e^{2\pi\xi w_1}$. 
We do the following transformations which do not change the determinant. 
Each remaining couple of double rows of the matrix $F$ has the form
$$R_s=(\mathbb{O},M(\xi-Q_s+(k_1+...+k_{Q_s})\tau), \mathbb{O}),$$  
$$R'_s=(\mathbb{O},M(\xi-Q_s-1+(k_1+...+k_{Q_s})\tau), \mathbb{O}).$$  
We replace the row $R'_s$ by $T_s=e^{-2\pi w_1}R_s-R'_s$ which is now free
from the terms containing $e^{2\pi \xi w_1}$. Next, we rearrange the rows of $F(\xi, l)$ in such a way that the new rows $T_s$ go after the first row of $F(\xi, l)$. This yields a rearrangment
of the rows of the matrix $(V_1,..., V_N,F_{N+1},..., F_{N+K})$ which after this rearrangment 
acquires a transparent block structure 
$$\begin{pmatrix} X & Y\\ Z & T \end{pmatrix},$$
where $N\times N$ matrix $X$ and $K\times N$ matrix $Z$
are independent on $\xi$, $T$ is a $K\times K$ lower triangular matrix while the $N\times K$ matrix $Y$ does
not contain terms with $e^{2\pi\xi w_1}$.

Since $T$ is lower-triangular the coefficient of $e^{2\pi Kw_1\xi}$ comes from the diagonal elements of the matrix $T$ only. 
The diagonal elements of the matrix $T$ are equal to $m_{N-1}(\xi+...)$ and have  non-zero coefficients in front of $e^{2\pi\xi w_1}$. 
Thus, it remains to prove that $\det X\neq 0$.

\medskip
We have
$$X=\begin{pmatrix}
u_1^K & u_2^K&\ldots& u^K_N\\
0 & u_2^{K}(e^{-2\pi w_2}-e^{-2\pi w_1}) & \ldots & u^K_N(e^{-2\pi w_N}-e^{-2\pi w_1})\\
0 & u_2^{K-q_2}(e^{-2\pi w_2}-e^{-2\pi w_1}) e^{2\pi w_2(q_2\tau - Q_2)}& \ldots & u^{K-q_2}_N(e^{-2\pi w_N}-e^{-2\pi w_1})e^{2\pi w_N(q_2\tau - Q_2)}\\
& & \ldots &\\
0 & (e^{-2\pi w_2}-e^{-2\pi w_1}) e^{2\pi w_2(K\tau - Q_{N-1})}& \ldots & (e^{-2\pi w_N}-e^{-2\pi w_1})e^{2\pi w_N(K\tau - Q_{N-1})}
\end{pmatrix}.$$

Note that since $\Re w_1 \ne \Re w_j, j > 1$ all the factors $e^{-2\pi w_j}-e^{-2\pi w_1}$ are non-zero. We have
$$\det X = u_1^K\prod_{j\neq l}(e^{-2\pi w_j}-u^{-2\pi w_l})\det
\begin{pmatrix}u_2^K &\ldots& u^K_N\\
u_2^{K-q_2} e^{2\pi w_2(q_2\tau - Q_2)} &\ldots& u^{K-q_2}_N e^{2\pi w_N(q_2\tau - Q_2)}\\
&\ldots& \\
 e^{2\pi w_2(K\tau - Q_{N-1})}&\ldots&  e^{2\pi w_N(K\tau - Q_{N-1})}
\end{pmatrix}=$$
$$u_1^K\prod_{j\neq l}(e^{-2\pi w_j}-e^{-2\pi w_l})\det X'.$$

Now we finally choose $Q_s$ in such a way that $1 \ll Q_2 \ll Q_3 \ll \ldots \ll Q_{N-1}$. This implies that $1 \ll q_2 \ll q_3 \ll \ldots \ll q_{N-1}$. Note that since all numbers $(q_2 + \ldots + q_s)\tau - Q_s$ are in $(0, \frac{1}{\alpha})$  corresponding exponents are uniformly bounded from above and from below. It remains to observe that, since $|u_2| > |u_3|> \ldots > |u_N|$, when we expand $\det X'$ as a sum over all permutations the diagonal term containing $u_2^Ku_3^{K-q_2}\ldots$ will dominate everything else and so $\det X'$ is non-zero. The lemma is proved.
$\square$

\subsection{Proof of Theorem \ref{compl}}
Modifying above arguments we can prove that the system $\gG(g;\alpha,\beta)$ is complete.
Moreover, since the function $m_0$ is a priori almost everywhere non-zero and the matrix $D(\xi, l)$ almost always has full rank we do not need the assumptions about $m_0$ and the irrationality of $\alpha \beta$. 
On the other hand looking at the proof of Theorem \ref{irrational} we can see that if $\alpha \beta >1$ then there exists an infinite-dimensional space of function which are orthogonal to all $\gG(g;\alpha,\beta)$.
Finally, if $\alpha \beta = 1$ then the frame operator is unitary equivalent to the multipication by the Zak transform and since in our case Zak transform is analytic and hence almost everywhere 
non-zero we have completeness in this case as well.

\subsection{Rational densities with large denominators}
\begin{remark} One can check that all arguments  from the
proof of Theorem \ref{irrth} remains true if $\alpha=\frac{p}{q}$ is a rational number with sufficiently big denominator $q\geq q(1-\alpha)$. In particular, this means that under the assumptions of Theorem \ref{irrational} there exists at most countable set of exceptional $\alpha$-s (such that $\gG(g;\alpha,1)$ is not a frame) with only one (possible) accumulating point $1$.
\label{remirr}
\end{remark}

Sometimes we are in the situation when there exists at most finite set of exceptional values, see Section \ref{hc}. 

\section{Other results\label{or}}

In this section we prove Theorems \ref{close1th}, \ref{close2th}, \ref{sdth}. In addition, we construct some counterexamples.

\subsection{Near the critical hyperbola (Theorem \ref{close1th})\label{hc}}
The first step in the proof repeat those in the proof of Theorem \ref{posth}. We also use notation introduced
in the proof of Therorem \ref{posth}.
 $$
 (LG)(\xi) = \summ_{s = 0}^{N-1} G\left(\xi + \frac{s}{\alpha}\right)m_s(\{ \xi\}),$$ 
 where $\{ t\}$ denotes the fractional part of $t$. It suffices to prove that (see Section \ref{s1p}, Step 1)
$$\|LG\|_2\gtrsim\|G\|_2,\quad G\in L^2(\mathbb{R}).$$
Now, we follow Step 2 and Step 3 (Sections \ref{st2}, \ref{st3}). It is enough to show that the following product of the Frobenius matrices tends to $0$ faster than 
some geometric progression:
\begin{equation}
\|F(p_{\{\xi\}})F(p_{\{\xi-\frac{1}{\alpha}\}})...F(p_{\{\xi-l\slash\alpha\}})\|\leq Cq^l,\quad \text { for some } q\in(0,1).
\label{gpineq}
\end{equation}
First we note that the spectrum of $F(p_\xi)$ coincides with the zero set of the polynomial
 $$p_\xi(z) = \frac{m_0(\xi) + m_1(\xi)z + \ldots + m_{N-1}(\xi) z^{N-1}}{m_{N-1}(\xi)}.$$ On the other hand, from the assumption of Theorem \ref{close1th} we have
$$\Re \mathcal{Z}(z,\xi)=\Re\frac{m_{N-1}(\xi)p_{\xi}(z)}{\prod_{k=1}^N(1-ze^{2\pi w_k\slash\alpha})}>0$$
and by the argument principle gives us that the number of zeroes of $p_\xi$ (counting with multiplicities) inside the unit disk $\mathbb{D}$ equals to $N$. 
Actually they are located in a smaller disk $\{z:|z|<\rho\}$, where $\rho<1$ is chosen so that 
$\Re\mathcal{Z}(\rho e^{2\pi i t},\xi)>0$,
$\xi\in[0,1]$, $t\in\mathbb{R}$, and $\rho>e^{-2\pi w_k\slash \alpha}$, $k=1,...,N$.

That means that the spectral radius of $F(p_\xi)$ is strictly less than $1$. In particular,
\begin{equation}
\|F^M(p_\xi)\|\leq q < 1 \text{ for sufficently large } M,
\label{srF}
\end{equation}
where $\|\cdot\|$ is an operator norm of matrix.
Moreover, this inequality is {\it uniform} with respect to $\xi\in[0,1]$ (numbers $q$ and $M$ does not depend on $\xi$). 

\medskip

Given this $M$ one can choose $\alpha$ sufficiently close to $1$ so that the matrices 
$F(p_{\{\xi-\frac{k}{\alpha}\}})$ and $F(p_{\{\xi-\frac{j}{\alpha}\}})$ are arbitary close to each other 
if $|k-j|<M$, so for $M\ll l$ the product in \eqref{srF}
can be represented as a product of $l\slash M_0$ uniformly strictly contractive matrices. This completes the proof. $\qed$

\medskip

Combining Remark \ref{remirr} and Theorem \ref{close1th} we get Theorem \ref{close2th}.

\medskip

\begin{remark} The condition $\Re \mathcal{Z}(z,\xi)>0$, $|z|=1$, $\xi\geq0$ can be reformulated as
$$\sum_{n\geq 0}m_0\biggl{(}\xi+\frac{n}{\alpha}\biggr{)}\cos(nt)>0,\quad \xi,t\in\mathbb{R}.$$
\label{r1}
\end{remark}
This form is useful if we want to check the inequality for all rescaled functions $g(\cdot\slash\beta)$ since rescaling of window $g$ 
corresponds to the rescaling of $m_0$.
\begin{proof}
Put $z=e^{i t}$. We have
$$\Re \mathcal{Z}(z,\xi)=\sum_{k=1}^Na_ke^{2\pi\xi w_k}\Re\frac{1}{1-ze^{2\pi w_k\alpha}}=$$
$$\sum_{k=1}^Na_ke^{2\pi\xi w_k}\sum_{n=0}^\infty\Re(z^ne^{2\pi n w_k\alpha})=\sum_{n\geq 0}m_0\biggl{(}\xi+\frac{n}{\alpha}\biggr{)}\cos(nt).$$
\end{proof}

Now, we are in position to prove Corollary \ref{corrconvex}. From the equation $m_0(\xi)=\hat{g}(\xi)$ we conclude that Fourier series
$$\sum_{n\geq 0}m_0\biggl{(}\xi+\frac{n}{\alpha}\biggr{)}\cos(nt)$$
has positive, convex coefficients for any $\xi$. It is known that such the Fouries series are positive, which
can be deduced by applying the Abel transform twice. $\qed$

\subsection{Lattices with large densities} In this section we prove Theorem \ref{sdth}. The proof is similar to one of Theorem \ref{irrational}.
By the standard rescaling we can assume $\beta=1$.
As in Section \ref{irrpr} it suffices to prove that the discrete operator $L_\theta$ defined by \eqref{Lmatrix} satisfies
$$\|L_\theta P_\theta\|\gtrsim\|P_\theta\|.$$

\medskip

Since $\alpha\leq\frac{1}{N}$, the operator $L_\theta$ can be split into the following $N\times N$ blocks (after possible omiting some rows):
$$B(\theta',N):=\begin{pmatrix}M(\theta')\\
M(\theta'-1)\\
...\\
M(\theta'-N)
\end{pmatrix}=\begin{pmatrix}m_0(\theta') & m_1(\theta') &.... &m_{N-1}(\theta')\\
m_0(\theta'-1) & m_1(\theta'-1) &.... &m_{N-1}(\theta'-1)\\
...\\
m_0(\theta'-(N-1))& m_1(\theta'-(N-1))&....&m_{N-1}(\theta'-(N-1))
\end{pmatrix},$$
where $\theta' = \theta + \frac{k}{\alpha} - n$ is such that $\theta' \ge N - 1$.

So,  Theorem \ref{sdth} follows from  
\begin{lemma}
Let $a_k\neq 0$, $k=1,...,N$ and $\Re w_k\neq\Re w_l$, $k\neq l$. Then
$$\det B(\theta,N)\neq0.$$
\end{lemma}
\begin{proof}
Put $y_k=e^{-2\pi w_k}$, $u_k=e^{2\pi w_k\slash\alpha}$, $A_k=a_ke^{2\pi\xi w_k}$, $k=1,...,N$. Then
$$B=\begin{pmatrix} A_1 & A_2 &...& A_{N}\\
A_1y_1 & A_2y_2 &...& A_Ny_N\\
&&....&\\
A_1y_1^{N-1} & A_2y_2^{N-1} &...& A_Ny_N^{N-1}\\
\end{pmatrix}
\begin{pmatrix}1 & -\sum_{k\neq 1}u_k &...& (-1)^{N-1}\prod_{k\neq 1}u_k\\
1 & -\sum_{k\neq 2}u_k &...&(-1)^{N-1}\prod_{k\neq 2}u_k\\
& &...&\\
1 & -\sum_{k\neq N}u_k&.. & (-1)^{N-1}\prod_{k\neq N}u_k
\end{pmatrix}= XY.
$$ 
Entries of the matrix $Y$ are symmetric polynomials with respect to subsets of variables $\{u_1,...,u_N\}$.
We have
$$\det Y = \pm\prod_{k,l,k\neq l}(u_k-u_l)\neq0.$$
On the other hand,
$$\det X=\prod_{k=1}^NA_k\det\begin{pmatrix} 1 & 1 &...& 1\\
y_1 & y_2 &...& y_N\\
&&....&\\
y_1^{N-1} & y_2^{N-1} &...& y_N^{N-1}\\
\end{pmatrix}=\prod_{k=1}^NA_k\det Z.$$
Vandermonde matrix $Z$ has non-zero determinant.
\end{proof}

\begin{proposition} For any $N$ there exists rational window $g$ of degree $N$ such that $\Re w_k\neq\Re w_l$, $k\neq l$ and Gabor system $\gG(g;1\slash(N-1),1)$ is
not a frame.
\label{nfprop}
\end{proposition}
\begin{proof}
Let us consider operator $L_\theta$ for some fixed $\theta$. Put $I=(....,1,1,1,1,....)^T$ (all entries of $I$ are equal to $1$).
If $L_\theta I=0$, then $L_\theta$ is not bounded away from zero as an operator on $\ell^2(\mathbb{Z})$ and, hence, $\gG(g;1\slash(N-1),1)$
is not a frame (see the proof of Theorem \ref{irrational}).
\smallskip
But if $\alpha=\frac{1}{N-1}$, then there exist only $N-1$ different rows in $L_\theta$. So, for any $\{w_k\}_{k=1}^N$ we can find a non-trivial sequence $\{a_k\}_{k=1}^N$
such that $LI=0$.
\end{proof}


\subsection{Counterexamples} Using the similar arguments as in the proof of Proposition \ref{nfprop} we can prove the following proposition.
\begin{proposition} For any rational number $\alpha$ there exists rational window $g$ such that Gabor system $\gG(g;\alpha,1)$
is not a frame.
\end{proposition}
On the other hand, if we carefully look at the proof of Theorem \ref{irrational} we can find a different method
of constructing non-frame Gabor systems. In particular, this leads us to a construction of Gabor non-frame systems 
with irrational densities.

\begin{theorem}
Let $f(x) = \sum\limits_{k = 1}^3 \frac{a_k}{x - i w_k}$. Assume that $\alpha > \frac{5}{6}$ and that for some $0.99 < \xi_0 < 1$ we have
\begin{equation}\label{condition}
m_0\biggl{(}\xi_0 - \frac{2}{\alpha} + 2\biggr{)} = m_1\biggl{(}\xi_0 - \frac{1}{\alpha} + 1\biggr{)} = m_2(\xi_0) = 0.
\end{equation}
Then the function $f$ does not give a frame for this $\alpha$.
\end{theorem}
\begin{proof}
By the Theorem \ref{mainth} a necessery and suficient condition for our system to be a frame is a bound
\begin{equation}\label{equivalent}
\intt_0^{1/\alpha}\summ_{n\in \Z}|\summ_{s = 0}^2 G\biggl{(}\xi - n - \frac{s}{\alpha}\biggr{)}m_s(\xi)|^2d\xi \gtrsim ||G||_{L^2(\R)}^2.
\end{equation}

Let us take the function $G$ being the characteristic function of $\delta$-vicinity of the point $\xi_0 - \frac{2}{\alpha}$. It is easy to see that  \eqref{equivalent} is not satisfied for small enough $\delta$ depending on $\eps$ even if we bound its left-hand side by 
\begin{equation}
\intt_0^{1/\alpha}\summ_{n\in \Z}3\summ_{s = 0}^2 \biggl{|}G(\xi - n - \frac{s}{\alpha})\biggr{|}^2|m_s(\xi)|^2d\xi 
\end{equation}
since for all $t$ such that $G(t)\ne 0$ corresponding $|m_s(\xi)|$ can be made smaller than any $\eps > 0$ if $\delta$ is small enough.
\end{proof}
\begin{remark}
Note that the similar result can be proved for the sum of arbitrary many kernels as long as $\alpha$ and $\xi$ are close enough to $1$ so that we don't have to worry about additional conditions coming from  $\xi \in (1, \frac{1}{\alpha})$.
\end{remark}

For fixed $\alpha, w_1, w_2, w_3$ conditions \eqref{condition} are three homogenious linear equations in $a_1, a_2, a_3$.
For this system to have a nontrivial solution corresponding $3\times 3$ determinant has to be zero. We will construct numbers $\alpha, w_1, w_2, w_3$ such that this determinant is zero and $\Re w_k$ are pairwise different. 
Moreover, in our construction we would have $w_2 = \frac{1}{2\pi}$, $w_3 = -\frac{1}{2\pi}$ and arbitrary $\alpha$ (rational or irrational) with $|\alpha - \frac{6}{7}| < \frac{1}{1000}$.

Let us begin with explicitly writing down the matrix corresponding to  \eqref{condition}
\begin{equation*}
\begin{pmatrix}
e^{2\pi w_1 (\xi _0- \frac{2}{\alpha} + 2)}&
e^{2\pi w_2 (\xi _0 - \frac{2}{\alpha} + 2)}&
e^{2\pi w_3 (\xi _0 - \frac{2}{\alpha} + 2)}\\
e^{2\pi w_1(\xi _0 - \frac{1}{\alpha} + 1)}(e^{\frac{2\pi}{\alpha}w_2} + e^{\frac{2\pi}{\alpha}w_3})&
e^{2\pi w_2(\xi _0 - \frac{1}{\alpha} + 1)}(e^{\frac{2\pi}{\alpha}w_1} + e^{\frac{2\pi}{\alpha}w_3})&
e^{2\pi w_3(\xi _0 - \frac{1}{\alpha} + 1)}(e^{\frac{2\pi}{\alpha}w_1} + e^{\frac{2\pi}{\alpha}w_2})\\
e^{2\pi w_1\xi _0 + \frac{2\pi}{\alpha}w_2 + \frac{2\pi}{\alpha}w_3}&
e^{2\pi w_2\xi _0 + \frac{2\pi}{\alpha}w_1 + \frac{2\pi}{\alpha}w_3}&
e^{2\pi w_3\xi _0 + \frac{2\pi}{\alpha}w_1 + \frac{2\pi}{\alpha}w_2}
\end{pmatrix}.
\end{equation*}

First of all we note that whether this determinant is zero or not does not depend on $\xi_0$ (in particular it is irrelevant if $0.99 < \xi_0 < 1$ or not). 
Therefore, without loss of generality we can assume that $\xi_0 = \frac{1}{\alpha} - 1$. Expanding the determinant (which we view as a function of $w_1$) and dividing it by $e^{\frac{2\pi}{\alpha}(w_1+w_2+w_3)}$ we get
$$
F(w_1) = e^{2\pi w_1}(e^{-2\pi w_2}- e^{-2\pi w_3}) - e^{-2\pi w_1}(e^{2\pi w_2} - e^{2\pi w_3}) - e^{\eps 2\pi w_1}(e^{-\eps 2\pi w_2} 
$$
$$
- e^{-\eps 2\pi w_3}) + e^{-\eps 2\pi w_1}(e^{\eps2\pi w_2} - e^{\eps 2\pi w_3}) + C,
$$

where $\eps = \frac{1}{\alpha} - 1$ and $C$ is a constant such that $F(w_2) = F(w_3) = 0$.

Let us first put $\alpha = \frac{6}{7}$. In that case $\eps = \frac{1}{6}$ and thus $e^{2\pi w_1} = (e^{\eps 2\pi w_1})^6$. Denoting $e^{\eps 2\pi w_1} = z$ and recalling that $w_2 = \frac{1}{2\pi}$, $w_3 = -\frac{1}{2\pi}$ the equation is rewritten as  
\begin{equation}
(e - e^{-1})(z^6 + z^{-6}) - (e^{1/6} - e^{-1/6})(z + z^{-1}) - e^2 + e^{-2} + e^{1/3} - e^{-1/3} = 0.
\end{equation}

This equation has solutions $z_1 = e^{1/6}\approx 1.18, z_2 = e^{-1/6}\approx 0.85$. But one can numerically verify that it also has negative solutions $z_3 \approx -1.12$, $z_4  \approx -0.89$.
 They correspond for example to $w_1 = 0.108 + 3i$ and $w_1 = -0.111 + 3i$ respectively. For any $\alpha$ close to $\frac{6}{7}$ we can find a close solution by the argument principle.

\medskip

So, we proved the following theorem.
\begin{theorem}
There exists a rational window $g$ of degree $3$ and irrational number $\alpha<1$ such that
$\gG(g;\alpha,1)$ is not a frame.
\label{nfirrth}
\end{theorem}

\section{Sampling in shift-invariant spaces \label{shiftinv}} 

The results on Gabor frames allow us to obtain new theorems on sampling in shift-invaraint spaces,
generated by rational windows. We follow mainly the pattern of \cite{Gro}, which in turns relies on the Janssen's version of duality theory.

First, we remind the basic definitions. Given a function $g$ which belongs to the Wiener amalgam space $W_0=W(\ell^1,\mathbb{C})$ i.e. $g$ is continuous
and
\begin{equation}
\|g\|_{W_0}=\sum_{k\in\mathbb{Z}}\max_{x\in[k,k+1]}|g(x)|<\infty.
\label{amalgam}
\end{equation}
Consider the shift-invariant space $V^2(g)$ which consists of the functions of the form
\begin{equation}
f(t)=\sum_{k\in\mathbb{Z}}c_kg(t-k),\quad \{c_k\}\in\ell^2(\mathbb{Z}).
\label{fV2g}
\end{equation}
Clearly, $V^2(g)\subset L^2(\mathbb{R})$ and
\begin{equation}
\|f\|_{L^2}\leq \|g\|_{W_0}\|\{c_k\}\|_2.
\label{fgWineq}
\end{equation}

We will assume the following stability of the generator $g$:

\begin{proposition}[see e.g. \cite{Ron}] 
The following properties are equivalent
\begin{enumerate}
\begin{item}
There exists $C>0$ such that 
\begin{equation}
\|\sum_{k\in\mathbb{Z}}c_kg(t-k)\|\geq C\|\{c_k\}\|_2;
\end{equation}
\end{item}
\begin{item}
\begin{equation}
\sum_{k\in\mathbb{Z}}|\hat{g}(\xi-k)|^2>0 \quad \text{ for all } \xi\in\mathbb{R}.
\label{fgineq}
\end{equation}
\end{item}
\end{enumerate}
\end{proposition}
We say that a sequence $\Lambda\subset\mathbb{R}$ is called {\it separated} if for some $\delta>0$, $|\lambda-\mu|>\delta$, $\lambda,\mu\in\Lambda,\lambda\neq\mu$.

A separated sequence $\Lambda$ is called {\it sampling} for $V^2(g)$ if there exist constants $A,B>0$ such that
\begin{equation}
A\|f\|^2_2\leq\sum_{\lambda\in\Lambda}|f(\lambda)|^2\leq B\|f\|^2_2,\quad f\in V^2(g).
\end{equation}

In what follows we apply these definitions to the sequences $\alpha\mathbb{Z}$, $\alpha<1$ and spaces $V^2(g)$ generated by the rational function $g$.

\medskip

The relation between sampling and frame properties is given by the following statement.

\begin{proposition} Let $g\in W_0$ has stable integer shifts. The following are equivalent:
\begin{enumerate}
\begin{item}
The family $\gG(g;\alpha,1)$ is a frame for $L^2(\mathbb{R})$;
\end{item}
\begin{item}
There exists $A,B>0$ such that for each $x\in[0,1]$ and $f$ of the form \eqref{fV2g}
\begin{equation}
A\|\{c_k\}\|^2_2\leq\sum_{m\in\mathbb{Z}}|f(x-\alpha m)|^2\leq B\|\{c_k\}\|^2_2.
\end{equation}
\end{item}
\end{enumerate}
\label{V2prop}
\end{proposition}

We refer the reader to (now) classical article \cite{Gro1} for the proof of this proposition
and also to \cite{Gro} for more general sequences.

\smallskip

We combine Proposition \ref{V2prop} with Theorems \ref{irrth} and \ref{irrational}.
\begin{lemma} Let the function $g$ have the form \eqref{eq:0_1} and, in addition $g\in W_0$
and $m_0(\xi)\neq0$, $\xi>0$. Then for each $\alpha\not\in\mathbb{Q}$, $\alpha\in(0,1)$ there exist $A_\alpha,B_\alpha>0$
such that
\begin{equation}
A_\alpha\|\{c_k\}\|^2_2\leq\sum_{k\in\mathbb{Z}}|f(\alpha k)|^2\leq B_\alpha\|\{c_k\}\|^2_2
\label{fcineq}
\end{equation}
for each function $f$ of the form $f(t)=\sum_kc_kg(t-k)$.
\end{lemma}
\begin{proof}
It suffices to prove that $g$ admits stable sampling, for example check the inequality \eqref{fgineq}. We have
$$
\hat{g}(\xi)=
\begin{cases}
\sum_{w_k<0}a_ke^{2\pi\xi w_k}, \quad \xi>0\\
-\sum_{w_k>0}a_ke^{2\pi\xi w_k}, \quad \xi<0
\end{cases}.
$$
each sum in the right hand-side has the leading term as $|\xi|\rightarrow\infty$ (it corresponds to the smallest $|w_k|$).
Thus, we have \eqref{fgineq}.
\end{proof}

\begin{theorem}
Let the function $g$ have the form \eqref{eq:0_1} and, in addition, $g\in W_0$ and $m_0(\xi)\neq0$, $\xi>0$.
Then, for each $\alpha\in(0,1)$ the set $\alpha\mathbb{Z}$ is a sampling for $V^2(g)$ i.e. there exist $A_\alpha,B_\alpha>0$
such that
\begin{equation}
A_\alpha\|\{c_k\}\|^2_2\leq\sum_{k\in\mathbb{Z}}|f(\alpha k)|^2\leq B_\alpha\|\{c_k\}\|^2_2,\quad f\in V^2(g).
\label{Vineq}
\end{equation}
\end{theorem}
\begin{proof}
The left-hand side inequality is a direct consequemce of \eqref{fcineq} and \eqref{fgWineq}. The proof of the right-hand side
inequality follows the classical Plancherel-Polya pattern. Each function $f\in V^2(g)$ has the form
\begin{equation}
f(t)=\sum_{k=1}^Na_k\sum_{n=-\infty}^{\infty}\frac{c_n}{t-(n+iw_k)}
\label{ff}
\end{equation}
and thus admits analytic continuation to $\mathbb{C}\setminus\cup_{k=1}^N(\mathbb{Z}+iw_k)$.
Denote
$$q(x)=1-e^{2\pi iz},\quad Q(z)=\prod_{k=1}^Nq(z-iw_k)$$
and consider an entire function of exponential type which in addition belongs to $L^2(\mathbb{R})$
\begin{equation}
F(z)=Q(z)f(z).
\end{equation}
Fucntion $F$ belongs to Paley-Wiener space $\mathcal{PW}_a$ for some $a$ and, hence,  satisfies Plancherel-Polya inequality
$$\sum_{k\in\mathbb{Z}}|F(\alpha k)|^2\lesssim\|F\|^2_2.$$
\end{proof}

\begin{remark}
It follows from \cite{LyubNes} that condition $\alpha\not\in\mathbb{Q}$ cannot be omitted in general. On the other hand the right inequality 
in \eqref{Vineq} is not related to irrationality of $\alpha$.
\end{remark}
\begin{remark}
Relation \eqref{fcineq} can be viewed as a statement on sampling by linear combination of values of functions in the Paley-Wiener space. Problems of such type appear in study of eigenfunction expansions of operator pencils, see e.g. \cite{L2}.
\smallskip
Indeed, given $f$ of the form \eqref{ff} we denote 
$$H(z)=q(z)\sum_{n=-\infty}^{\infty}\frac{c_n}{z-n}.$$
We than have
$$H\in\mathcal{PW}_{[0,1]},\quad \|H\|_2\asymp\|\{c_k\}\|_{\ell^2},$$
and \eqref{Vineq} can be read as 
$$\sum_{j=-\infty}^{\infty}|\sum_{k=1}^Na_kq(\alpha j-iw_k)^{-1}H(\alpha j-iw_k)|^2\asymp\|H\|^2_2.$$
\end{remark}

\end{document}